\documentclass{amsart}

\usepackage{amsmath}
\usepackage{amsfonts}
\usepackage{amssymb,amscd,amsthm}
\usepackage{geometry}
\usepackage{mathrsfs}
\usepackage[bookmarksnumbered, colorlinks, linktocpage, plainpages]{hyperref}
\usepackage{nameref}
\usepackage{wasysym,amssymb,eufrak,indentfirst,color}
\usepackage{extarrows}

\usepackage{enumitem}

\usepackage{mathtools}
 \usepackage{appendix}
\usepackage{extarrows}
\usepackage{setspace}

\usepackage{cases}
\usepackage{graphicx}
\usepackage{ragged2e}

\numberwithin{equation}{section}
\newcommand\asertion[1]{ssertion $({\mathrm{\romannumeral #1\relax}})$}%假设1

\newcommand{\asselm}{associative element}%定义结合元
\newcommand{\almost}{para-}%定义可裂

\newcommand{\Hom}{\text{Hom}}

\newcommand{\bfs}{Without loss of generality we can assume}%定义可裂

\newcommand{\spc}{\mathbb{C}}
\newcommand{\spr}{\mathbb{R}}
\newcommand{\spo}{\mathbb{O}}%八元数

\newcommand\huaa[1]{\mathscr{A}(#1)}%结合元集
\newcommand\hua[3]{\mathscr{#1}^{#2}(#3)}%循环元集
\newcommand\pureim[1]{\mathit{Im}(#1)}%纯虚元集
\newcommand{\re}{\text{Re}\,}%

\newcommand\fx[2]{\left<#1,#2\right>}%

\newcommand\fsh[1]{\left|\left|#1\right|\right|}%

\def\O{\mathbb{O}}
\def\R{\mathbb{R}}

\def\abs#1{\left|#1\right|}

\newcommand\clifd[1]{C\ell_{#1}}%结合元集

\newtheorem{mydef}{Definition}[section]
\newtheorem{rem}[mydef]{Remark}
\newtheorem{eg}[mydef]{Example}
\newtheorem{cor}[mydef]{Corollary}
\newtheorem{prop}[mydef]{Proposition}
\newtheorem{lemma}[mydef]{Lemma}
\newtheorem{thm}[mydef]{Theorem}

\newtheorem{step }[stp]{Step }

\begin{document}

\title{Para-linearity as the nonassociative counterpart of linearity}
\author{Qinghai Huo}
\email[Q.~Huo]{hqh86@mail.ustc.edu.cn}
\address{Department of Mathematics, University of Science and Technology of China, Hefei 230026, China}

\author{Guangbin Ren}
\email[G.~Ren]{rengb@ustc.edu.cn}
\address{Department of Mathematics, University of Science and Technology of China, Hefei 230026, China}
 
 \date{\today}


\thanks{This work was supported by the NNSF of China (11771412).}

	\begin{abstract}
In  an octonionic Hilbert space  $H$, the octonionic linearity is taken to fail for the maps induced by the octonionic inner products, and it should be  replaced with the octonionic para-linearity.
However, to introduce the notion of the octonionic para-linearity we encounter  an insurmountable obstacle. That is,
  the axiom
$$\left\langle pu ,u\right\rangle=p\left\langle u ,u\right\rangle$$ for any octonion $p$ and element $u\in H$
introduced by Goldstine and Horwitz    in 1964
 can not be interpreted as  a property to be obeyed  by the   octonionic para-linear maps.
In this article,   we solve this  critical  problem by showing  that this axiom is in fact non-independent from others.
This enables us to initiate  the study  of  octonionic para-linear  maps.
We can thus establish 	
the octonionic Riesz representation theorem  which, up to isomorphism,  identifies    two  octonionic Hilbert spaces with  one being  the dual of the other.
 The dual space consists of continuous left \almost linear functionals and it becomes a right $\O$-module under the  multiplication defined in terms of
 the second associators which measures the failure of $\O$-linearity.  This right multiplication has an alternative expression $${(f\odot p)(x)}=pf(p^{-1}x)p,$$
  which is
  a generalized   Moufang identity.
Remarkably, the multiplication  is   compatible with the  canonical norm, i.e.,
	$$\fsh{f\odot p}=\fsh{f}\abs{p}.$$
Our final conclusion is  that para-linearity is the nonassociative counterpart of linearity.
  	\end{abstract}
 
\maketitle

\tableofcontents

	\section{Introduction}

What is   the nonassociative counterpart of linearity is
a natural question in the nonassociative realm.
This question has no answer even  in octonionic functional analysis
since there still lacks the octonionic   Riesz representation theorem which depends on the  maps abstracted  from the inner products.
In order to explore the reason for this lack, we start with the octonionic Hilbert theory.

The theory of octonionic Hilbert spaces
was initiated    by 	Goldstine and Horwitz \cite{goldstine1964hilbert}   in 1964 and
 has  many applications in the spectral theory \cite{ludkovsky2007Spectral},  operator theory \cite{ludkovsky2007algebras}, and physics \cite{deleo1996oqm,gunaydin1973quark,gunaydin1976OHilbert,Rembielinski1978tensorOHilbert}.
 But Compared with the quaternionic  Hilbert space theory
 \cite{horwitz1993QHilbertmod,razon1992Uniqueness,
  razon1991projection,soffer1983quaternion,viswanath1971normal}, the development of the octonionic Hilbert space theory is lagging behind.
    Nonetheless, the existing research on the {octonionic} Hilbert space theory has emerged some effective methods.
    For example, Goldstine and Horwitz \cite{goldstine1964hilbert}
     {attribute} the study of the {octonionic} Hilbert space to   the associative setting;
  see \cite{goldstine1964hilbert2,Saworotnow1968generalize} for the further development.
Ludkovsky and Spr\"ossig  \cite{ludkovsky2007algebras,ludkovsky2007Spectral} find the theory can get well developed if the $\O$-Hilbert space admits   the tensor decomposition  of a real Hilbert space with  the algebra of octonions $\O$.

Now we  focus on the definition of the  {octonionic} Hilbert space.

	\begin{mydef}[\cite{goldstine1964hilbert}]\label{def:gold}
	An $\O$-Hilbert space	  $H$ is a left $\O$-module with an  $\O$-inner product  $$\left\langle\cdot,\cdot \right\rangle :H\times H \rightarrow \mathbb{O}$$
		such that $(H, \re \left\langle\cdot,\cdot \right\rangle)$ is a real Hilbert space. Here the $\O$-inner product satisfies the following axioms for all $ u,v\in H$ and $p\in \O$:
			\begin{enumerate}[label=(\alph*)]
			\item  $\left\langle u+v,w\right\rangle=\left\langle u,w\right\rangle+\left\langle v,w\right\rangle$;
						\item  $\left\langle u ,v\right\rangle=\overline{\left\langle v ,u\right\rangle}$;
			\item  $\left\langle u ,u\right\rangle\in \spr^+$;  and $\left\langle u ,u\right\rangle=0$ if and only if $u=0$;
			\item $\left\langle tu ,v\right\rangle=t\left\langle u ,v\right\rangle$ for $t\in \R$;
			\item $\re \left\langle pu ,v\right\rangle=\re (p\left\langle u ,v\right\rangle)$;
			\item $\left\langle pu ,u\right\rangle=p\left\langle u ,u\right\rangle$.
		\end{enumerate}

	\end{mydef}
	
The functionals that need to be studied in the {octonionic} Hilbert space are naturally induced by the octonionic inner product
\begin{eqnarray}\label{eq:paralinear-095}
	f(x)=\langle x, v \rangle\end{eqnarray}
for any given element $v\in H$.  However, these real linear maps in  \eqref{eq:paralinear-095} are generally
  not octonionic linear  {maps} since   there may hold
$$f(px)\neq pf(x)$$
for some $p\in\O$ and $x\in H$.

 It is still an unsolved problem so far
how to define an appropriate functional in octonionic functional analysis. This is because in Definition \ref{def:gold} Axiom $(f)$, as a property of $\O$-inner-product, cannot be abstractly interpreted  into the property of a functional on a general $\O$-normed space.

Fortunately, we can overcome this  obstacle   in a roundabout way in this article.
 We find that
 Axiom $(f)$ is actually non-independent. This   motivates us to introduce a kind of maps beyond  $\O$-linearity, called   $\O$-para-linear maps,
   defined in  any $\O$-modules.

 An \textbf{$\O$-\almost linear} map in a left $\O$-module $M$
 is defined to be a real linear map $f\in \Hom_\R(M,\O)$ subject to the condition
\begin{eqnarray}\label{eq:def-para-104}\re [p,x,f]=0\end{eqnarray} for all $p\in \O$ and $x\in M$, where the element $$[p,x,f]:=f(px)-pf(x)$$
is called \textbf{the second associator} related to $f$.
Condition  \eqref{eq:def-para-104} becomes    Axiom (e) in Definition  \ref{def:gold}
 in the specific case where   $f$ is a map induced by the octonionic inner product as in \eqref{eq:paralinear-095}.

The set of \almost linear functions in a left $\O$-module constitutes  a right $\O$-module under  the multiplication
	\begin{eqnarray}\label{eq:f p}
	(f\odot  p)(x):=f(x)p-[p,x,f].
	\end{eqnarray}

Similarly,  $\mathbb K$-\almost linearity  can also be defined for any real normed division algebra
$$\mathbb K=\mathbb R, \spc, \mathbb{H}, \O,$$
where $\mathbb{H}$ is the algebra of quaternions.
    Since $\mathbb K$-\almost linearity  degenerates to linearity in the associative setting, this   shows      \almost linearity
is a  nonassociative counterpart of   linearity.

 $\mathbb K$-\almost linearity can be used to define  $\mathbb K$-Hilbert spaces.
In the specific case where  $\mathbb K=\O$, this definition  is belonging to
Goldstine and Horwitz \cite{goldstine1964hilbert} except that a non-independent axiom  is removed.

\begin{mydef} Suppose that  $H$ is a real Hilbert space under the real inner product $\langle \cdot,\cdot\rangle_{\mathbb  R}$
	and also a left $\mathbb{K}$-module. We call $H$ a \textbf{ $\mathbb{K}$-Hilbert space} if there exists an $\spr $-bilinear map $$\left\langle\cdot,\cdot \right\rangle :H\times H \rightarrow \mathbb{K}$$  such that
	$$Re \left\langle\cdot,\cdot \right\rangle=\langle \cdot,\cdot\rangle_{\mathbb  R}$$
	and
	satisfying
	\begin{enumerate}
		\item \textbf{($\mathbb{K}$-\almost linearity)} $\left\langle\cdot,u\right\rangle$ is left $\mathbb{K}$-\almost linear for all $u\in H$.
				\item \textbf{($\mathbb{K}$-hermiticity)} $\left\langle u ,v\right\rangle=\overline{\left\langle v ,u\right\rangle}$ for all $ u,v\in H$.
		\item \textbf{(Positivity)} $\left\langle u ,u\right\rangle\in \spr^+$;  and $\left\langle u ,u\right\rangle=0$ if and only if $u=0$.
	\end{enumerate}
\end{mydef}

 Now we can use    $\mathbb O$-\almost linearity to study  the Riesz representation theorem in   $\O$-Hilbert spaces  $H$.
 The theorem should be expressed in the following form
 \begin{eqnarray}\label{eq:iso-riesz-832}
	H^{*}\cong {H}.
\end{eqnarray}
Here  $H^*$ is the dual space  consisting  of $\O$-para-linear functions and
the bijection  in  \eqref{eq:iso-riesz-832}  between  two spaces is both an isometric isomorphism as  real Banach spaces and a conjugate isomorphism as  $\O$-modules.

Of the most difficult   is  to show  the compatibility
between the module structure of $H^*$ and  its  canonical  norm, i.e.,
	$$\fsh{f\odot p}=\fsh{f}\abs{p}$$ for all $f\in H^*$ and $p\in \O$,   where  $f\odot p$ is the right multiplication defined in \eqref{eq:f p}.
 %	by $$(f\odot p)(x)=f(x)p-[p,x,f].$$
 Overcoming this difficulty requires another expression  of   the right multiplication
	$${(f\odot p)(x)}=pf(p^{-1}x)p,$$
which  is a  generalized  Moufang identity \cite{schafer2017introduction}.

 The octonionic Riesz representation theorem in \eqref{eq:iso-riesz-832}  can be restated as
  \begin{eqnarray}\label{eq:iso-riesz-852}
	(H^{*})^-\cong {H}.
\end{eqnarray}
Here $(H^{*})^-$ denotes the set $H^{*}$ endowed
  with a left $\O$-module defined by
	$$p\cdot x:=x\overline{p}.$$
  The point is that
the  bijection  in  \eqref{eq:iso-riesz-852}  is  an  isomorphism  of two left $\O$-modules in contrast to the conjugate  isomorphism in \eqref{eq:iso-riesz-832}.

 Space $(H^{*})^-$ can be endowed with an $\O$-inner product via   isomorphism
\eqref{eq:iso-riesz-852}  so that it becomes an $\O$-Hilbert space.
 The octonionic Riesz representation theorem shows that
 the two $\O$-Hilbert  spaces in \eqref{eq:iso-riesz-852}   coincide
up to   isomorphism. This result is true if $\O$ is replaced by any normed division algebras $\mathbb K$.

\begin{thm}[Riesz representation theorem] Let $\mathbb K$
be a normed division algebra.
 Let  $H$  be  a $\mathbb K$-Hilbert space and  $H^*$   its dual consisting  of all continuous $\mathbb K$-para-linear functionals. Then we can identify the two $\mathbb K$-Hilbert spaces $H$ and $(H^*)^-$
up to   isomorphism.  \end{thm}

Opposed  to the associative setting, we have two distinct spaces
$$ H^{*_\spo} \subsetneq  H^{*}.$$
Here  $H^{*_\spo}$   consists of all continuous $\O$-linear functionals of $H$. Under   isomorphism \eqref{eq:iso-riesz-832},  $H^{*_\spo}$ is corresponding to the nucleus of $H$ consisting of associative elements. Symbolically,
 \begin{eqnarray*}
 	H^{*_\spo}\cong \huaa{H}.
 	\end{eqnarray*}
 This shows that the class of $\O$-linear functionals is too small to represent every element in an $\O$-Hilbert space.

We remark that there are profound relations between the   associative and nonassociative theories.
For example,  the algebra of octonions is closely related to   the Clifford algebra of  $6,7$ dimensions \cite{harvey1990spinors}, and the category of left $\spo$-modules is shown to be isomorphic to the category of left $\clifd{7}$-modules \cite{huo2021leftmod}.
 But
 the  octonionic theory has its intrinsic objects such as the $\O$-para-linear functionals as demonstrated in the octonionic  Riesz representation theorem.

After     para-linearity is introduced as a substitute for linearity in nonassociative framework, one can further study para-linear operators,
weak topology related to para-linear functionals,
octonionic   $C^*$   algebra,  and para-linear connections.

	\section{Preliminaries}\label{sec:preliminary}
	In this section we review 	
	some notations and basic properties about  octonions  $\O$ and  left $\O$-modules. We refer to \cite{baez2002octonions,huo2021leftmod}
	for more details.
	\subsection{Octonions}\label{subsec:O}
	The set of octonions $\spo$  is a nonassociative, noncommutative, normed division algebra over  $\spr$.   Let     $$e_0=1, \ e_1, \ \dots, \ e_7$$ be a basis of $\O$ as a real vector space subject to the 	  multiplication rule
	 	\begin{align}\label{eq:epsilon notation-348}
	&e_ie_j=\sum_{k=1}^7\epsilon_{ijk}e_k-\delta_{ij}.
	\end{align}
for any $i,j=1,\dots,7$. 	Here $\delta_{ij}$ is the Kronecker delta  and $\epsilon_{ijk}$ is   completely antisymmetric  with value 1 precisely when $$ijk = 123, 145, 176, 246, 257, 347, 365.$$

Let   $x=x_0+\sum_{i=1}^7x_ie_i$ 	be an  octonion with
	all $ x_i\in\spr$.
We define 	its conjugate
 $$\overline{x}:=x_0-\sum_{i=1}^7x_ie_i,$$
 its norm  $$|x|=\sqrt{x\overline{x}},$$
 and its real part     $$\re{x}=x_0=\frac{1}{2}(x+\overline{x}).$$
  	\subsection{$\O$-modules}
 	 We   recall  some basic notations and  results on $\O$-modules in this subsection.	
 We refer to
\cite{jacobson1954structure,Shestakov1974rightrep,
Shestakov2016bimod,huo2021leftmod}
  for     octonionic modules and even for  modules of  alternative algebras.

A left  $\O$-module  $M$ is
		a  real vector space $M$ endowed with an $\O$-scalar multiplication such that    the \textbf{left associator} is    left  alternative, i.e.,
  for all  $p,q\in\O$ and $  m\in M$ we have
	\begin{eqnarray}\label{eq:left ass}
	[p,q,m]=-[q,p,m].
	\end{eqnarray}
Here  the left associator is defined by
	$$[p,q,m]:=(pq)m-p(qm).$$	
Condition  \eqref{eq:left ass} is  equivalent to the requirement
	$$r(rm)=r^2m $$
for all 	$r\in \O$ and all $ m\in M$.
	
Of the most important identity involving  associators in a  left $\O$-module    $M$ is  the five-term identity
\begin{equation}\label{eq:intrdct ass id}
[p,q,r]m+p[q,r,m]=[pq,r,m]-[p,qr,m]+[p,q,rm]
\end{equation}
for all $p,q,r\in \O$ and $m\in M$.
This even holds when $M$ is  a   nonassociative algebra \cite{schafer2017introduction}.

In the theory of   left $\O$-modules,
associative and conjugate associative elements play key roles
(\cite{huo2021leftmod}). Let $M$ be a left $\O$-module.
An element $m\in M$ is called  an \textbf{associative element}  if
	$$(pq)m=p(qm)$$
and called
  a \textbf{conjugate associative element}  if
	$$(pq)m=q(pm).$$
The  set of all associative elements is called
		  the \textbf{nucleus}  of $M$ and denoted by $\huaa{M}$, i.e.,  $$\huaa{M}:=\{m\in M\mid [p,q,m]=0\text { for all } p,q \in \O\}.$$
	 The  set of all  conjugate associative elements  of  $M$ is denoted by $\hua{A}{-}{M}$, i.e.,    $$\hua{A}{-}{M}:=\{m\in M\mid (pq)m=q(pm) \text { for all } p,q \in \O\}.$$

	\begin{thm}[\cite{huo2021leftmod}]\label{thm:M=OA+OA^-}
		Let $M$ be a left $\O$-module. Then  $$M\cong\spo\huaa{M}\oplus {\spo}\hua{A}{-}{M}.$$
	\end{thm}

%	A real vector space $M$ is called a left (right) $\O$-module if the left (right) associator is  left (right) alternative:
%	$$[p,q,m]=-[q,p,m]\quad \text{ for all } p,q\in\O,\, m\in M,$$	
%	which is equivalent to the condition:
%	$$r(rm)=r^2m \text{ for all }r\in \O,\, m\in M.$$
%	Where the \textbf{left associator} is defined by
%	$$[p,q,m]=(pq)m-p(qm).$$
%	A left  $\O$-module $M$ is called an  \textbf{ $\O$-bimodule} (definition see \cite{huoqinghai2020bimodule})  if  the  \ass\ is alternative:
%	$$[p,q,m]=[m,p,q]=[q,m,p]$$
%	for all $p,q\in \O$, and $m\in M$.
	%Similarly we can define the right associator and middle associator on $\O$-bimodules.	
	
%In a left $\O$-module $M$,	the following identity is very important
%	\begin{equation}\label{eq:[p,q,r]m+p[q,r,m]=[pq,r,m]-[p,qr,m]+[p,q,rm]}
%	[p,q,r]m+p[q,r,m]=[pq,r,m]-[p,qr,m]+[p,q,rm]
%	\end{equation}
%	for all $p,q,r\in \O$ and all $m\in M$. %Where  the \textbf{left associator} of $M$  is given by $[p,q,m]:=(pq)m-p(qm)$.

\section{$\spo$-\almost linear functions}\label{sec:almost linearity}

The notion of the  second   associator
 is used to describe the extent to which a real linear map  is far away from the octonionic linear map.
This notion also allows us to define
 the  \almost linear map.

\begin{mydef} Let $M$ be a left $\O$-module and let    $ \Hom_{\mathbb R}(M,\O)$   be the set of all real linear functions $f: M \to\O$.
We introduce   a real trilinear map
$$[\cdot, \cdot, \cdot]: \O\times M\times \Hom_{\mathbb R}(M,\O)\xlongrightarrow[]{} \O,$$
defined by
\begin{equation}\label{def:sec-asso-294}[p,x,f]:=f(px)-pf(x).\end{equation}
  We call this map  \textbf{the second associator} in  $M$, which measures the failure of $\O$-linearity.

\end{mydef}

\begin{mydef} Let $M$ be a left $\O$-module.
A real linear map $f: M \xlongrightarrow[]{} \O$
is called an \textbf{$\O$-\almost linear} function if it satisfies
$$\re [p,x,f]=0$$
for any $p\in\O$ and $x\in M$.
 \end{mydef}

We denote by 	$\Hom_{L\O}(M,\O)$ the set of all $\O$-\almost linear functions.  By definition, we have

$$\Hom_{L\O}(M,\O)=\{
f\in \Hom_{\mathbb R}(M,\O): \  \re [p,x,f]=0 \ \
\mbox{for \  all \ $p\in\O$ \ and  \ $x\in M$ }
\}.
$$

 \begin{eg}\label{eg:Rp} The right multiplication on $\O$ provides a typical example of  $\O$-\almost linear function, i.e.,
 $$R_q\in  \Hom_{L\O}(\O,\O)$$ for any $q\in \O$.
Here $R_q$ denotes the
 right multiplication   defined by $$R_q(x)=xq.$$
 In this typical example, we find that the associator and the second associator coincide, i.e.,
\begin{equation}\label{eq:two-alt-993}[p,x,R_q]=[p,x,q] \end{equation}
 for any $p, x, q\in \spo$.
 Indeed, by definition we have
 $$[p,x,R_q]=R_q(px)-pR_q(x)=(px)q-p(xq)=[p,x,q]. $$
Identity \eqref{eq:two-alt-993}  also   implies    $$\re [p,x,R_q]=0 $$
so that $R_q$ is $\O$-\almost linear.	
\end{eg}

\bigskip

   If  $f\in \Hom_\spr(M,\O)$, then it can be   written  as    \begin{eqnarray}\label{eq:def-exp-rea-746}f(x)=f_{\R}(x)+\sum_{i=1}^7
 e_if_i(x),\end{eqnarray} where  $$f_{\R}(x):=\re f(x)$$ and  $f_i(x)\in \R\text{ for } i=1,\dots,7$.

 Now we give a criterion for determining whether a real linear function  $f\in \Hom_\spr(M,\O)$ is an  $\O$-\almost linear map.

\begin{thm}\label{thm: f in Hom(M,M')  equivalent:}
   Let $f\in \Hom_\spr(M,\O)$ with expression as in \eqref{eq:def-exp-rea-746}. Then
	  $f$ is $\O$-\almost linear if and only if
\begin{eqnarray}\label{eq:rel-real-346}f_i(x)=f_{\R}(\overline{e_i}x)\end{eqnarray} for all  $i=0,\dots ,7$ and all $x\in\O$.
	\end{thm}

\begin{proof} We assume $f$  $\O$-\almost linear.
By definition, we have
\begin{align*}
f_\R(\overline{e_i}x)&=\re f(\overline{e_i}x)\\
&=\re (\overline{e_i}f(x))\\
&=\re \Big(\overline{e_i}\sum_{j=0}^7 e_jf_j(x)\Big)\\
&=f_i(x).
\end{align*}
%	Let $A_p(x,f)=\sum e_jy_j$, then
%	\begin{align*}
%		y_j&=-\mathit{Re}(e_jA_p(x,f))=-\mathit{Re}(e_j(f(px)-pf(x)))\\
%		&=-\mathit{Re}\left(f(e_j(px))-e_j(pf(x))\right)\\
%		&=-\mathit{Re}\left(f((e_jp)x-[e_j,p,x])-(e_jp)f(x)+[e_j,p,f(x)]\right)\\
%		&=-\mathit{Re}\left(A_{e_jp}(x,f)-f[e_j,p,x]+[e_j,p,f(x)]\right)\\
%		&=f_{\R}([e_j,p,x]).
%		\end{align*}
%This is the desired conclusion.

Conversely, under assumptions  \eqref{eq:rel-real-346}, for any $i=1, \ldots, 7$ we claim
\begin{eqnarray}\label{eq:calim-093}
e_if(x)
= e_if_{\R}(x)+\sum_{j,k=1}^7
 \epsilon_{ijk}e_k f_j(x)- f_i(x)
\end{eqnarray}
Indeed, by calculation we have
\begin{eqnarray*}
e_if(x)&=&e_i(f_{\R}(x)+\sum_{j=1}^7
 e_jf_j(x))
 \\
 &=& e_if_{\R}(x)+\sum_{j=1}^7
 e_ie_jf_j(x)
 \\
 &=& e_if_{\R}(x)+\sum_{j,k=1}^7
 \epsilon_{ijk}e_k f_j(x)- f_i(x)
\end{eqnarray*}
and the claim holds. From the claim we have
\begin{eqnarray*}\re [e_i, x, f]&=&Re(f(e_i x)-e_if(x))
\\
 &=& f_{\R}(e_ix)+f_i(x)=0.
\end{eqnarray*}
This finishes the proof.

\end{proof}

The second associator of a $\O$-para-linear map in $\Hom_{L\O}(M,\O)$ has a close relation with the associator in $\O$.
\begin{thm}\label{thm: f in Hom(M,M')  equivalent-874}
   Let $f\in \Hom_\spr(M,\O)$ with expression as in \eqref{eq:def-exp-rea-746}. Then
	  $f$ is $\O$-\almost linear if and only if
\begin{eqnarray}\label{eq:rel-real-7 }
[p,x,f]=\sum_{i=1}^7
	e_i f_{\R}([e_i,p,x])\end{eqnarray} for all  $i=0,\dots ,7$ and all $p\in \O$ and  $x\in M$.
	\end{thm}

\begin{proof} To prove the necessity, we might as well set  $p=e_i$ with $i=1,\ldots, 7$. Since $f$ is $\O$-\almost linear, Theorem \ref{thm: f in Hom(M,M')  equivalent:} shows
\begin{eqnarray}\label{eq:para-iden-23}f(x)=f_{\R}(x)-\sum_{j=1}^7 e_j f_{\R}(e_j x).\end{eqnarray}
Notice that
\begin{eqnarray}\label{eq:para-iden-263} e_j(e_ix)=(e_je_i)x-[e_j,e_i,x]=\sum_{k=1}^7\epsilon_{jik}e_k x-\delta_{ij}x-[e_j,e_i,x].\end{eqnarray}
Now we replace $x$ by $e_ix$ in \eqref{eq:para-iden-23} and then apply \eqref{eq:para-iden-263}. Due to the real linearity of $f_{\R}$ we have
\begin{eqnarray}\label{eq:para-iden-293}
f(e_ix)=f_{\R}(e_ix)
-\sum_{j,k=1}^7 \epsilon_{jik}e_jf_{\R}(e_k x)+e_i f_{\R}(x)+\sum_{j=1}^7 e_j f_{\R}([e_j, e_i,x]).
 \end{eqnarray}

On  the other hand we have shown in \eqref{eq:calim-093}  that
\begin{eqnarray}\label{eq:calim-433}
e_if(x)
= e_if_{\R}(x)+\sum_{j,k=1}^7
 \epsilon_{ijk}e_k f_j(x)- f_i(x)
\end{eqnarray}
 We subtract both sides in \eqref{eq:para-iden-293}  and \eqref{eq:calim-433}. Since $\epsilon_{ijk}$ is   completely antisymmetric, we apply the fact that
$$f_i(x)=-f_{\R}( {e_i}x)$$
 and identity  \eqref{eq:rel-real-346}  to conclude
{\begin{eqnarray*}\label{eq:rel-real-876}
[e_i,x,f]=f(e_ix)-e_if(x)=\sum_{j=1}^7
	e_j f_{\R}([e_j,p,x]).\end{eqnarray*}}

Conversely, if conditions \eqref{eq:rel-real-7 }  hold, then $$\re[p, x, f]=0$$ so that $f$ is $\O$-para-linear by definition.
 \end{proof}		

\begin{rem}\label{rem:linear coincide}
Theorems \ref{thm: f in Hom(M,M')  equivalent:} and \ref{thm: f in Hom(M,M')  equivalent-874} hold if    $\O$ is replaced by any Cayley-Dickson algebra $\mathbb{K}$ (see \cite{harvey1990spinors})
under  the similar   notion of  $\mathbb{K}$-\almost linear functions.
  $\mathbb{K}$-\almost linearity becomes   $\mathbb{K}$-linearity	when  $\mathbb{K}$ is   an associative normed division algebra because    the   associators in
    \eqref{eq:rel-real-7 }
   vanish in this setting. As a result, the notion of \almost linearity is a canonical generalization of the classical linearity from  the associative  to the nonassociative realm.

\end{rem}

As direct consequences of
Theorems \ref{thm: f in Hom(M,M')  equivalent:} and \ref{thm: f in Hom(M,M')  equivalent-874},
 the second associator  vanishes for associative element and an $\spo$-\almost linear map is  uniquely determined by its real part.

\begin{cor}\label{cor: A_p(x,f)=0 and f(px)=pf(x)}
Assume that  $f:M\to \O$ is $\O$-\almost linear.
\begin{enumerate}
	\item For each \asselm\ $x\in \huaa{M}$, we have $$[p,x,f]=0$$ and hence for all $p\in \O$ $$f(px)=pf(x).$$
	\item  $f_{\R}=0$  if and only if  $f= 0$.
	\end{enumerate}

\end{cor}

\bigskip

To study  the $\O$-module structure of $\Hom_{L\O}(M,\O)$, we need  some useful  identities about the second associators.
	\begin{prop}\label{prop: left,bimod associator } Suppose that
	   $f\in \Hom_{L\O}(M,\O)$.  Then for any   $p,q\in \spo$ and $x\in M$ we have
	\begin{align}
 	p[p,x,f]&=[p,x,f]\overline{p}=[p,\overline{p}x,f];
 \label{eq:Omod left bi pAp(xf)=Ap(xf)p}
  \\
	\re\left(f([p,q,x])\right)&=\re \left([p,x,f]{q}\right); \label{eq:Omod left bi Re(f[pqx])=Re(Ap(x,f)q)}\\
	\re\left([p,x,f]q\right)&=-\re\left([q,x,f]p\right).\label{eq:Omod left bi Re(Ap(xf)q=-Re(Aq(xf)p))}
	\end{align}	
	%	\begin{enumerate}
	%		\item $\re\left(f([p,q,x])\right)=\re \left(A_p(x,f)\overline{q}\right)$;
	%		\item $pA_p(x,f)=A_p(x,f)\overline{p}$;
	%		
	%		\item $\re\left(A_p(x,f)q\right)=-\re\left(A_q(x,f)p\right)$.
	%	\end{enumerate}
\end{prop}
\begin{proof}
	As usual,	we write  $$f(x)=f_{\R}(x)+\sum_{i=1}^7 e_if_i(x).$$
	%where $f_j(x)\in \re M',\ j=0,1,\dots,7$.

	To prove
	\eqref{eq:Omod left bi pAp(xf)=Ap(xf)p}, without loss of generality we may assume that  $p\in \pureim{\spo}$. Let $p=\sum_{i=1}^7 p_ie_i$ with $ p_i\in \spr$. Then by Theorem \ref{thm: f in Hom(M,M')  equivalent-874}, we have
	\begin{align*}
	p[p,x,f]&=\sum_{i,j=1}^7 p_ip_je_i[e_j,x,f]\\
	&=\sum_{i,j,k=1}^7 p_ip_j e_i e_kf_{\R}([e_k,e_j,x]).
	\end{align*}
%	$$p[p,x,f]=\sum_{i,j=1}^7 p_ie_i[p_je_j,x,f]=\sum_{i,j,k=1}^7 p_ip_j e_i e_kf_{\R}([e_k,e_j,x]),$$
	Similarly, we get
	\begin{align*}
	[p,x,f]\overline{p}&=\sum_{i,j=1}^7 p_ip_j[e_j,x,f]\overline{e_i}\\
	&=-\sum_{i,j,k=1}^7 p_ip_je_ke_if_{\R}([e_k,e_j,x]).
	\end{align*}
	%$$[p,x,f]\overline{p}=\sum_{i,j=1}^7 [p_je_j,x,f]\overline{p_ie_i}=-\sum_{i,j,k=1}^7 p_ip_je_ke_if_{\R}([e_k,e_j,x]).$$
	Since $$e_i e_k+e_ke_i=-2\delta_{ik},$$
	it follows that
	$$p[p,x,f]-[p,x,f]\overline{p}=-2\sum _{i,j=1}^7 p_ip_jf_{\R}([e_i,e_j,x]).$$
	Because of the skew symmetricity of the associator, we get
$$p[p,x,f]=[p,x,f]\overline{p}.$$
By direct calculations, we get
\begin{eqnarray*}
[{p},\overline{p}x,f]&=&f(\abs{p}^2x)-pf({\overline{p}}x)\\
&=&\abs{p}^2f(x)-pf({\overline{p}}x)\\
&=&p(\overline{p}f(x))-pf({\overline{p}}x)\\
&=&-p[\overline{p},x,f]\\
&=&p[{p},x,f].
\end{eqnarray*}
%\begin{eqnarray*}\label{eqpf:Ar(rxf)=r-Ar(xf)}
%[{p},px,f]&=&-
%[\overline{p},px,f]
%\\
%&=&-f(\abs{p}^2x)+{\overline{p}}f(px)\notag\\
%&=&-(\abs{p}^2f(x)-{\overline{p}}f(px))\notag\\
%&=&-(\overline{p}(pf(x))-{\overline{p}}f(px))\notag
%\\
%&=&\overline{p} [ {p},x,f]
%\\
%&=&-\overline{p} [\overline{p},x,f]
%\\
%&=& -[\overline{p},x,f]  {p}
%\\
%&=& [ {p},x,f]  {p}.
%\end{eqnarray*}

	Now we come to prove  \eqref{eq:Omod left bi Re(f[pqx])=Re(Ap(x,f)q)}.
	without loss of generality we may assume   $q=e_i$ with   $i=1,\dots,7$.
	By direct calculations, we have
	\begin{align*}
	\re \left([p,x,f]{q}\right)&=\re\Big(\sum_{j=1}^7 (e_jf_{\R}([e_j,p,x]))e_i\Big)\\
	&=\re\Big(\sum_{j=1}^7 f_{\R}([e_j,p,x]) e_je_i\Big)\\
	&=-f_{\R}([e_i,p,x])\\
	&=\re\big(f([p,q,x])\big).
	\end{align*}
 	
	Identity \eqref{eq:Omod left bi Re(Ap(xf)q=-Re(Aq(xf)p))} follows from \eqref{eq:Omod left bi Re(f[pqx])=Re(Ap(x,f)q)} directly.
	%It  follows directly from \eqref{eq:Omod left bi Re(f[pqx])=Re(Ap(x,f)q)} that
	%	\eqref{eq:Omod left bi Re(Ap(xf)q=-Re(Aq(xf)p))} holds true.
	%	 It's clear that both sides are 0 if $p$ or $q$ is real number. Let  $p=e_i,\ q=e_j$, then    by the same method as above, we have,
	%	\begin{align*}
	%	\re\big(A_{e_i}(x,f)e_j\big)&=\re \left(\sum e_ke_jf_{\R}[e_k,e_i,x]\right)=-f_{\R}[e_j,e_i,x]
	%	\intertext{and}
	%	\re\big(A_{e_j}(x,f)e_i\big)&=-f_{\R}[e_i,e_j,x]=f_{\R}[e_j,e_i,x]
	%	\end{align*}
	%	This proves \eqref{eq:Omod left bi Re(Ap(xf)q=-Re(Aq(xf)p))} in case $p=e_i,\ q=e_j$, the general case follows.
\end{proof}
\begin{rem}
Identity \eqref{eq:Omod left bi pAp(xf)=Ap(xf)p}	generalizes the third identity of Lemma 3.9 in  \cite{Grigorian2017octonionbundles}.
\end{rem}

Utilizing Proposition \ref{prop: left,bimod associator }, we can  endow $\Hom_{L\O}(M,\O)$ with  a canonical  right $\spo$-module structure.
Let $r\in \O$,   $f\in \Hom_{L\O}(M,\O)$, and $x\in M$. We define  the right scalar multiplication   \begin{eqnarray}\label{eqdef:<x,fr>}
	(f\odot r)(x)\coloneqq\ f(x)r-[r,x,f].
	\end{eqnarray}

	\begin{thm}\label{Thm:left, bimod Hom(M,M')}
	If  $M$ be a left $\spo$-module, then $\Hom_{L\O}(M,\O)$ is a   right $\spo$-module.\end{thm}
\begin{proof}
	First we need to show that for any
 $f \in \Hom_{L\O}(M,\O)$
and  $r\in \O$ we have $$f\odot r\in \Hom_{L\O}(M,\O).$$
That is to prove 	\begin{eqnarray}\label{eq:re-iden-125}\re [p,x,f\odot r]=0.\end{eqnarray}
By direct calculations, we have %\eqref{eq:algmod A_p(x,fr)=[p,f(x),r]+A_p(x,f)r-A_r(px,f)+pA_r(x,f)}
	\begin{align}
\label{eq:algmod A_p(x,fr)=[p,f(x),r]+A_p(x,f)r-A_r(px,f)+pA_r(x,f)}
	[p,x,f\odot r]
%		&=\left<px,fr\right>-p\fx{x}{fr}\notag\\
&=(f\odot r)(px)-p((f\odot r)(x))\notag\\
%		&=\fx{px}{f}r-A_r(px,f)-p\left(f(x)r-A_r(x,f)\right)\notag\\
	&=f(px)r-[r,px,f]-p\left(f(x)r-[r,x,f]\right)\notag\\
		&=\left(pf(x)+[p,x,f]\right)r-[r,px,f]-p\left(f(x)r-[r,x,f]\right)\notag\\
		&=[p,f(x),r]+[p,x,f]r-[r,px,f]+p[r,x,f].
	\end{align}
Since    $$\re [p,f(x),r]=0,\quad \re [r,px,f]=0,$$
	 \eqref{eq:algmod A_p(x,fr)=[p,f(x),r]+A_p(x,f)r-A_r(px,f)+pA_r(x,f)}  tells us that we have to  show
	$$\re([p,x,f]r)=-\re(p[r,x,f]).$$
	This can be easily deduced  by applying  identity \eqref{eq:Omod left bi Re(Ap(xf)q=-Re(Aq(xf)p))} and the fact  that $$\re (pq)=\re(qp)$$ for all $p,q\in \O$.
	
	Next we come to show  that   $\Hom_{L\O}(M,\O)$ is  a right $\spo$-module. It suffices to prove $$(f\odot r)\odot r=f\odot (r^2)$$
	for any $r\in \O$.
Due to  the uniqueness as shown in
a\asertion{2} of  Corollary \ref{cor: A_p(x,f)=0 and f(px)=pf(x)}, we need to prove
$$\re {\big((f\odot r)\odot r-f\odot (r^2)\big)(x)} =0.$$

By definition \eqref{eqdef:<x,fr>} and identity \eqref{eq:Omod left bi pAp(xf)=Ap(xf)p}, we have
%	\begin{align}
%	\fx{x}{(fr)r}&=\fx{x}{fr}r-A_r(x,fr)\notag\\
%	&=\fx{x}{f}r^2-A_r(x,f)r-A_r(x,fr)\label{eq:<x,fr,r>}
%	\intertext{and }
%	\fx{x}{fr^2}		&=	\fx{x}{f}r^2-A_{r^2}(x,f).\label{eq:<xfr2>}
%	\end{align}
	\begin{align}
{((f\odot r)\odot r)(x)}&={\big((f\odot r)(x)\big)}r-[r,x,f\odot r]\notag
\\
&=f(x)r^2-[r,x,f]r-[r,x,f\odot r]\label{eq:<x,fr,r>}
\intertext{and }
{(f\odot (r^2) )(x)}		&=	f(x)r^2-[{r^2},x,f].\label{eq:<xfr2>}
\end{align}
	Notice that
	 $\re [r,x,f\odot r]=0$, $\re [{r^2},x,f]=0$,
and
\begin{eqnarray}\label{eq:ide-127} [r,x,f]=-[\overline{r},x,f].\end{eqnarray}
	We thus conclude  from    \eqref{eq:<x,fr,r>} and  \eqref{eq:<xfr2>} that
	\begin{align*}
	\re {\big((f\odot r)\odot r-f\odot (r^2)\big)(x)}&
=\re([\overline r,x,f]r)=\re([\overline r,  rx,f])=0.
	\end{align*}
The last but one  step used \eqref{eq:Omod left bi pAp(xf)=Ap(xf)p}.
		This completes the proof.	
\end{proof}

	\section{$\O$-Hilbert spaces}
We adopt the definition of $\O$-Hilbert spaces
by Goldstine and Horwitz \cite{goldstine1964hilbert} except that we describe it in terms of $\O$-\almost linearity.
In addition, we find a non-independent axiom given by  Goldstine and Horwitz \cite{goldstine1964hilbert}.
	\subsection{Definition of  Hilbert $\spo$-modules}

We first introduce  a succinct definition of $\O$-Hilbert spaces.
This definition also  suits    all normed division algebras.
	
\begin{mydef}\label{def:hilbert O space}
	A left $\spo$-module $H$ is called a \textbf{ pre-Hilbert left $\spo$-module} if there exists an $\spr $-bilinear map $\left\langle\cdot,\cdot \right\rangle :H\times H \rightarrow \O$,  referred to as an \textbf{$\spo$-inner product}, satisfying:
	\begin{enumerate}
		\item \textbf{($\O$-\almost linearity)} $\left\langle\cdot,u\right\rangle$ is left $\spo$-\almost linear for all $u\in H$.
		%	\item $\mathit{Re}\left\langle pu ,v\right\rangle=\mathit{Re}\left\langle u ,\overline{p}v\right\rangle$ for any $p\in \spo,\ u,v\in H$.
		\item \textbf{(Octonion hermiticity)} $\left\langle u ,v\right\rangle=\overline{\left\langle v ,u\right\rangle}$ for all $ u,v\in H$.
		\item \textbf{(Positivity)} $\left\langle u ,u\right\rangle\in \spr^+$;  and $\left\langle u ,u\right\rangle=0$ if and only if $u=0$.
	\end{enumerate}
\end{mydef}

Under   axioms in Definition  \ref{def:hilbert O space}, we can deduce in Lemma \ref{lem:the second associator of H 's prop-943} below  that
\begin{equation}\label{eq:def-redu-403}\left\langle pu,  u\right\rangle=p\left\langle u, u\right\rangle
\end{equation}
for all $p\in \O$ and all $u\in H$. This equation  is one of     the axioms  of Goldstine and Horwitz as   part $(f)$ of   \cite[Postulate $2$]{goldstine1964hilbert}.

We introduce   a real trilinear map
$$[\cdot, \cdot, \cdot]: \O\times M\times M\xlongrightarrow[]{} \O,$$
defined by
 $$[p,u,v]:=\left\langle pu ,v\right\rangle-p\left\langle u ,v\right\rangle$$ for $u,v\in H$ and $p\in \O$. We call this map  \textbf{the second associator} in  $H$.

Identity \eqref{eq:def-redu-403}  can be restated as
  	$$[p,u,u]=0.$$

\begin{rem}\label{rem:ass coincide}
The second associator   has been introduced in the space {$\Hom_{L\O}(H,\O)$} in Section \ref{sec:almost linearity}.
These two notions in $H$ and $H^*$ coincide if we identify  $H$  with $H^*$ (see Theorem \ref{thm:Riesz representation theorem}). In fact, for any $p\in\O$ and $u, v\in H$ we have
\begin{eqnarray}\label{eq:ass coincide}
 [p,u,v]=[p,u,v'],\end{eqnarray}
where $v'\in   {\Hom_{L\O}(H,\O) }$ is a map associated to $v$ via the inner product, i.e.,
$$v'=\left\langle \cdot ,v\right\rangle.$$
Indeed, by definition we have
\begin{eqnarray*}
[p,u,v']&=&v'(pu)-pv'(u)
\\
&=& \left\langle pu ,v\right\rangle-p\left\langle u ,v\right\rangle
\\
&=& [p,u,v].
\end{eqnarray*}

\end{rem}

\begin{eg} If  $\O$ is regarded  as an $\O$-Hilbert space with the inner product
$$\fx{x}{y}:=x\overline{y}$$
then its second associator of $p, u, v$ is given by
\begin{eqnarray*} [p, u, v]&=&\fx{pu}{v}-p\fx{u}{v}
\\
&=& (pu)\overline{v}-p(u\overline v),
\end{eqnarray*}
which is exactly the associator of $p, u,   \overline v$ in the algebra $\O$.

\end{eg}

Second associators in an $\O$-Hilbert space can be studied with the aid of
the theory   of $\O$-para-linear functionals
based on relation \eqref{eq:ass coincide}.

\begin{lemma}
	Let $H $ be  a pre-Hilbert (left) $\spo$-module, $p, q\in \O$ and $v,u\in H$. Then we have
	\begin{align}
	%\fx{px}{u}&=&p\fx{x}{u}+A_p(x,u);\\
	[p,v,u]&=\sum_{i=1}^7 e_i \re \fx{[e_i,p,v]}{u};\label{eq:Ap(u,v)=sum ei re (<[e_i,p,u],v>}\\	
[{pq},v,u] &=\fx{[p,q,v]}{u}-[p,q,\fx{v}{u}]+p[q,v,u]+[p,qv,u];	\label{eq:Apq(v,u)=<[p,q,v],u>-[p,q,<v,u>]+pAq(v,u)+Ap(qv,u) }\\	
	\re\fx{[p,q,v]}{u}&=-\re\left(p[q,v,u]\right). \label{eq:Re<[p,q,x],y>=Re(pA(x,y))}
	\end{align}
	
	%		
	%	  Recall Theoreom \ref{Thm:left, bimod Hom(M,M')}, we konw $H^*$ is a right $\spo$-mod. So it make sense to consider $y'r$, where $y\in H,\ r\in \spo$.
	
\end{lemma}

\begin{proof} We fix $u\in H$ and  take  $f:=u'\in {\Hom_{L\O}(H,\O)}$ in
Theorem
\ref{thm: f in Hom(M,M')  equivalent-874}.
 Since  $$f(x)=u'(x)=\langle x, u\rangle$$
and
$$f_{\mathbb R}(x)=\re \langle x, u\rangle,$$
we thus conclude that
 identity \eqref{eq:Ap(u,v)=sum ei re (<[e_i,p,u],v>} follows from
Theorem
\ref{thm: f in Hom(M,M')  equivalent-874}.

	Identity \eqref{eq:Apq(v,u)=<[p,q,v],u>-[p,q,<v,u>]+pAq(v,u)+Ap(qv,u) } can be proved in the manner similar to the proof of identity \eqref{eq:intrdct ass id}. Identity \eqref{eq:Re<[p,q,x],y>=Re(pA(x,y))}  follows from \eqref{eq:Apq(v,u)=<[p,q,v],u>-[p,q,<v,u>]+pAq(v,u)+Ap(qv,u) } by taking the real part on both sides.
\end{proof}

If $H $ is a pre-Hilbert (left) $\spo$-module, then $(H,\left\langle \cdot,\cdot\right\rangle _{\spr})$ is a real Hilbert space,
where  $$\left\langle x,y\right\rangle _{\spr}:=\re\left\langle x,y\right\rangle.$$
Conversely, it is quite natural to  ask if  a real Hilbert space with a structure of  left $\O$-module  admits an $\O$-inner product.
\begin{thm}\label{thm:Hilbert R O}
	Let $H$ be a left $\O$-module equipped with a real valued inner product $\fx{\cdot}{\cdot}_0$. Then $H$ admits an $\O$-inner product $\fx{\cdot}{\cdot}$ such that
\begin{equation}\label{eq:real i.p. benzhi-942}\fx{\cdot}{\cdot}_\R=\fx{\cdot}{\cdot}_0\end{equation}
	if and only if
	\begin{equation}\label{eq:real i.p. benzhi}
	\left\langle px,y\right\rangle _0=\left\langle x,\overline{p}y\right\rangle _0
	\end{equation}
	holds
	for all $p\in \O$ and all $x,y\in H$.
	
\end{thm}
\begin{proof}
		Suppose the  $\O$-inner product $\fx{\cdot}{\cdot}$  in $H$ satisfies \eqref{eq:real i.p. benzhi-942}.
To prove  \eqref{eq:real i.p. benzhi}, we need to show
$$\re \fx{px}{y}=\re \left\langle x,\overline{p}y\right\rangle$$
for all $p\in \O$ and $x,y\in H$.
By direct calculation, we have
  	\begin{eqnarray*}\label{eqpf:Hilbert O=R 0}
	\re \fx{px}{y}&=&\re (p\fx{x}{y}+[p,x,y])
\\ &=&\re (p\fx{x}{y})
\\
  &=&  \re (\overline{p\fx{x}{y}})
  \\ &=& \re ({\fx{y}{x}}\overline{p})
  \\ &=& \re (\overline{p}{\fx{y}{x}})
	\\
&=&\re {\fx{\overline{p}y}{x}}
\\
&=&\re \left\langle x,\overline{p}y\right\rangle.
	\end{eqnarray*}

	Conversely,
	suppose that \eqref{eq:real i.p. benzhi} holds.
We claim that the following inner product satisfies all conditions:
	\begin{eqnarray}\label{eqpf:<>O def}
	\left\langle x,y \right\rangle_{\spo} :=\left\langle x,y\right\rangle _0-\sum_{i=1}^7 \left\langle e_ix,y\right\rangle _0e_i.
	\end{eqnarray}
According to the construction,  equation \eqref{eq:real i.p. benzhi-942} obviously holds.
	It remains to show that
$\left\langle \cdot,\cdot\right\rangle _\O$ is an $\spo$-inner product.
	
Evidently, $\left\langle \cdot,\cdot\right\rangle_{\spo} $ is $\spr$-bilinear. We come to prove the \almost linearity.
	Define\begin{equation}\label{eq:inner-prod-048}f(x):=\fx{x}{y}_\O\end{equation} for any given $y\in H$. Then $f$ is clearly a real linear map. We write  \begin{equation}\label{eq:inner-prod-058}f(x)=f_{\R}(x)+\sum_{i=1}^7 f_i(x)e_i,\end{equation}  where $f_\R(x),\, f_i(x)\in \R$.
%This together  another expression  of $f$ which comes from
Combining   \eqref{eqpf:<>O def}  with  \eqref{eq:inner-prod-048}, we obtain another expression  of $f$ which demonstrates that
 $$f_i(x)=-\left\langle e_ix,y\right\rangle _0
 =-\re \left\langle e_ix,y\right\rangle=-\re f(e_ix)=-f_{\mathbb R}(e_ix)$$ for $i=1,\dots,7$.  Then  Theorem \ref{thm: f in Hom(M,M')  equivalent:} implies that $f$ is left \almost linear. This proves the \almost linearity of $\left\langle \cdot,\cdot\right\rangle_{\spo} $.
	
	The octonionic hermiticity follows from  condition \eqref{eq:real i.p. benzhi}. Indeed,
	\begin{align*}
	\overline{{\fx{x}{y}}_\O}&=\left\langle x,y\right\rangle _0+\sum_{i=1}^7 \left\langle e_ix,y\right\rangle _0e_i\\
	&=\left\langle y,x\right\rangle _0+\sum_{i=1}^7 \left\langle x,\overline{e_i}y\right\rangle _0e_i\\
	&=\left\langle y,x\right\rangle _0-\sum_{i=1}^7 \left\langle {e_i}y ,x\right\rangle _0e_i\\
	&=\fx{y}{x}_\O.
	\end{align*}
	
Finally, we show the positivity of $\left\langle \cdot,\cdot\right\rangle_{\spo} $. For each $i=1,\dots, 7$, we have	$$\fx{e_ix}{x}_0=\fx{x}{\overline{e_i}x}_0=-\fx{x}{e_ix}_0=-\fx{e_ix}{x}_0$$
	so that $\fx{e_ix}{x}_0=0$. This means  $$\fx{x}{x}_{\spo}=\fx{x}{x}_0,$$ which  implies the positivity.	
\end{proof}

The  associator $[p,q,u]$  in the $\O$-Hilbert space $H$
  is
  an   skew-adjoint  map as a function of    $u$,
  while the second associator $[p,u,v]$  in the $\O$-Hilbert space $H$
  is an  alternating real bilinear map   of $u$ and $v$.

\begin{lemma}\label{lem:the second associator of H 's prop-943}
	Let $H$ be a pre-Hilbert left $\spo$-module. For any $u,v\in H$ and  $p,q\in \spo$, we have
	\begin{eqnarray}
	\left\langle [p,q,u],v\right\rangle _\R&=&-\left\langle u,[p,q,v]\right\rangle _\R; \label{eq:<[qpx]y>0=<[qpy]x>}\\			
{	[p,u,v]} &=&-[p,v,u].\label{eq:Ap(u,v)=-Ap(v,u)}
	\end{eqnarray}
\end{lemma}

\begin{proof}
	Identity  \eqref{eq:<[qpx]y>0=<[qpy]x>}
	 follows from Theorem \ref{thm:Hilbert R O}. Indeed,
	\begin{eqnarray*}
		\left\langle [p,q,u],v\right\rangle _\R&=&\left\langle (pq)u-p(qu),v\right\rangle _\R\\
		&=&\big\langle u,\overline{(pq)}v-\overline{q}(\overline{p}v)\big\rangle _\R\\
		&=&\left\langle u,[\overline{q},\overline{p},v]\right\rangle _\R\\
&=&\left\langle u,[q,p,v]\right\rangle _\R.		
\\
&=&-\left\langle u,[p,q,v]\right\rangle _\R.
	\end{eqnarray*}

	Due to  identities \eqref{eq:Ap(u,v)=sum ei re (<[e_i,p,u],v>} and \eqref{eq:<[qpx]y>0=<[qpy]x>} we obtain
	\begin{eqnarray*}
	[p,u,v]&=&\sum_{i=1}^7 e_i\fx{[e_i,p,u]}{v}_\R\\
		&=&-\sum_{i=1}^7 e_i\fx{[e_i,p,v]}{u}_\R\\
		&=&-[p,v,u].
	\end{eqnarray*}
	This proves \eqref{eq:Ap(u,v)=-Ap(v,u)}.
	
	\end{proof}

As a direct consequence of Lemma
\ref{lem:the second associator of H 's prop-943}, we have for any $p\in \O$ and $u\in H$
$$[p, u, u]=0.$$

The next lemma {tells} us  how it works when factors of octonions pass through the $\O$-inner product.

\begin{lemma}\label{lem:the second associator of H 's prop}
	Let $H$ be a pre-Hilbert left $\spo$-module. For any $u,v\in H$ and  $p,q\in \spo$, the following identities hold:
	\begin{eqnarray}	\fx{u}{pv}&=&\fx{u}{v}\overline{p}+[p,u,v];\label{eq:<u,pv>=<u,v>p^+Ap(uv)}\\
	\fx{pu}{qv}&=&p(\fx{u}{v}\overline{q})+p[q,u,v]+[p,u,qv];\label{eq:<pu,qv>=p(<u,v>q^- +pAq(u,v)+Ap(u,qv))}\\
	\fx{pu}{qv}&=&(p\fx{u}{v})\overline{q}+[pq,u,v]+\fx{[p,q,v]}{u}.\label{eq:<pu,pv>in lemma}
	\end{eqnarray}
\end{lemma}

\begin{proof}
	 	We first  prove \eqref{eq:<u,pv>=<u,v>p^+Ap(uv)}. By the octonionic hermiticity,  we conclude from \eqref{eq:Ap(u,v)=-Ap(v,u)} that
	\begin{eqnarray*}
		\fx{u}{pv}&=&\overline{\fx{pv}{u}}\\
		&=&	\overline{p\fx{v}{u}+[p,v,u]}\\
		&=&\fx{u}{v}\overline{p}-{[p,v,u]}\\
		&=&\fx{u}{v}\overline{p}+[p,u,v].
	\end{eqnarray*}		
{
		 \eqref{eq:<pu,qv>=p(<u,v>q^- +pAq(u,v)+Ap(u,qv))}
	can be deduced   from \eqref{eq:<u,pv>=<u,v>p^+Ap(uv)} directly. Indeed,
	\begin{eqnarray*}
		\fx{pu}{qv}&=&p\fx{u}{qv}+[p,u,qv]\notag\\
		&=&p(\fx{u}{v}\overline{q})+p{[q,u,v]}+[p,u,qv].%\label{eq:<pu,qv>=p(<u,v>q^- +pAq(u,v)+Ap(u,qv))}
	\end{eqnarray*}
	It follows from identities  \eqref{eq:Apq(v,u)=<[p,q,v],u>-[p,q,<v,u>]+pAq(v,u)+Ap(qv,u) } and \eqref{eq:Ap(u,v)=-Ap(v,u)} that
		$$p[q,u,v]+[p,u,qv]=\fx{[p,q,v]}{u}-[p,q,\fx{v}{u}]+[pq,u,v].$$
	Combining this with  \eqref{eq:<pu,qv>=p(<u,v>q^- +pAq(u,v)+Ap(u,qv))}, we get
	\begin{eqnarray*}
		\fx{pu}{qv}&=&p(\fx{u}{v}\overline{q})+\fx{[p,q,v]}{u}-[p,q,\fx{v}{u}]+[pq,u,v]\\
		&=&p(\fx{u}{v}\overline{q})+[p,\fx{u}{v},\overline{q}]+\fx{[p,q,v]}{u}+[pq,u,v]\\
		&=&(p\fx{u}{v})\overline{q}+[pq,u,v]+\fx{[p,q,v]}{u}.
	\end{eqnarray*}
}	
This  completes the proof.
\end{proof}

\begin{rem}
	 All   discussions in this subsection  still hold when   $\O$ is replaced by an arbitrary  real normed division algebra $\mathbb{K}$.
	This verifies    that the Hilbert space over any normed division algebras can be defined  uniformly.
 \end{rem}

\subsection{Norm of $\O$-Hilbert space}
In this subsection, we collect some known results about    the norm of an $\O$-Hilbert space; see \cite{goldstine1964hilbert}.
Our method is suitable for
  any Hilbert space over any normed division algebras.

\begin{mydef}%[\cite{huoqinghai2020nonass}]
	A normed left $\spo$-module is a pair $(X,\|\cdot\|)$ consisting of a left  $\spo$-module $M$ and a map
	$\|\cdot\|:\  X\rightarrow  \spr$
		satisfying the following axioms.
	\begin{enumerate}
		\item $\fsh{x} \geqslant 0 $ for all $x\in X$ with equality if and only if $x = 0$.
		\item $\fsh{p x}=|p|\fsh{x}$ for all $p\in \spo$ and $x\in X$.
		\item $\fsh{x+y}\leqslant\fsh{x}+\fsh{y}$ for all $x,y\in X$.
	\end{enumerate}

A normed $\O$-module  is said to be a Banach $\O$-module if it is complete with respect to its canonical   norm.
\end{mydef}

\begin{mydef} An \textbf{$\spo$-Hilbert space}
	is a  pre-Hilbert left $\spo$-module	$H$ which is complete with respect to  the norm
	$\fsh{\cdot}:H\rightarrow \mathbb{R}^+ $ induced by the $\O$-inner product, i.e.,
	\begin{eqnarray}\label{def:fshu ||u||}
	\fsh{u}=\sqrt{\fx{u}{u}}.
	\end{eqnarray}
	
\end{mydef}

It is easy to check that in an   $\spo$-Hilbert space there holds  the   polarization identity:
\begin{eqnarray}
\fx{u}{v}=\frac{1}{4} \sum_{i=0}^7e_i \left(\fsh{e_iu+v}^2-\fsh{e_iu-v}^2\right).
\end{eqnarray}
This means that the $\O$-inner-product can be recovered from the norm.

\begin{thm}
	Let $H$ be a Hilbert left $\spo$-module.
 Then 	$H$ is a Banach left $\O$-module with respect to the norm  induced from the $\O$-inner product.
 \end{thm}
\begin{proof}	
		Since  $(H,\fsh{\cdot})$ is a real Banach space, it remains to  show that  $$\fsh{pu}=\abs{p}\fsh{u}$$
for any $p\in \spo$ and $x\in X$.
According to  \eqref{eq:Ap(u,v)=-Ap(v,u)} and  \eqref{eq:<pu,pv>in lemma} we obtain
	\begin{align*}
	\fx{pu}{pu}&=(p\fx{u}{u})\overline{p}+[{pp},u,u]+\fx{[p,p,u]}{u}=|p|^2\fsh{u}^2.
	\end{align*}
	This  proves that $H$ is a Banach left $\O$-module.
 	\end{proof}

The Cauchy-Schwarz inequality has been established  in \cite{goldstine1964hilbert}.  Here we give an alternative   proof
based on Lemma \ref{lem:the second associator of H 's prop}.
\begin{prop}
	For all $u,v\in H$,  we have  the Cauchy-Schwarz inequality
\begin{eqnarray}\label{eq:Cauchy-Schwarz inequality}
\abs{\fx{u}{v}}\leqslant\fsh{u}\fsh{v}.
\end{eqnarray}
Equality holds in \eqref{eq:Cauchy-Schwarz inequality}  if and only if $u$ and $v$ are $\O$-linearly dependent.

\end{prop}

\begin{proof}	
		To prove the Cauchy-Schwarz inequality,
	without loss of generality  we may assume $\fsh{v}=1$. Denote $p=\fx{u}{v}$. It follows from  \eqref{eq:<u,pv>=<u,v>p^+Ap(uv)} that
	\begin{eqnarray*}
		\fx{u-pv}{pv}_\R&=&\re(\fx{u-pv}{v}\overline{p})\\
		&=&\re((\fx{u}{v}-p\fx{v}{v})\ \overline{p})\\
		&=&0.
	\end{eqnarray*}
	This implies that
	\begin{eqnarray}\label{eqpf:cs}
	\fsh{u}^2=\fsh{u-pv}^2+\fsh{pv}^2\geqslant \fsh{pv}^2.
	\end{eqnarray}
	Since  $$\fsh{pv}=\abs{p}\fsh{v}=\abs{p},$$
	 we can rewrite  \eqref{eqpf:cs}   as
	$$\abs{\fx{u}{v}}^2\leqslant\fsh{u}^2= \fsh{u}^2\fsh{v}^2.$$
	Moreover, in view of \eqref{eqpf:cs}, we conclude that $\abs{\fx{u}{v}}=\fsh{u}\fsh{v}$ if and only if
	$$\fsh{u-pv}^2=0,$$ which means  $u,v$ are $\spo$-linearly dependent.
\end{proof}

\section{Riesz representation theorem}\label{sec:O Hilbert}

Let $H$ be an $\O$-Hilbert space. We consider
  $$H^{*}:=\{f\in {\Hom_{L\O}(H,\O)}\mid f \text{ is continuous}\}.$$ For any $f\in H^{*}$,
we define the right scalar multiplication by \begin{eqnarray}\label{eqdef:<x,fr>-129}
	(f\odot r)(x)\coloneqq\ f(x)r-[r,x,f]
	\end{eqnarray}
and define the norm
\begin{eqnarray}\label{eqdef:funct norm}
\fsh{f}=\sup_{\fsh{x}\leqslant1}\fsh{f(x)}.
\end{eqnarray}
This makes  $H^{*}$     a Banach right $\O$-module.
Indeed, Theorem \ref{Thm:left, bimod Hom(M,M')} shows that $H^*$ is a right $\O$ module.
To show it is a normed space, we need to prove
$$\fsh{f\odot p}=\fsh{f}\ \abs{p}$$
for all $f\in H^*$ and $p\in \O$.
This needs an alternative expression of the right multiplication in \eqref{eqdef:<x,fr>-129}  below.

\begin{lemma}
	For any  $f\in H^*$ and $p\in \spo$, we have
	\begin{equation} \label{eq:<x,Tp>=pT(p^-1x)p}
	{(f\odot p)(x)}=pf(p^{-1}x)p.
	\end{equation}
\end{lemma}
\begin{proof}
\bfs \ $|p|=1$. This means $p^{-1}=\overline{p}$. By the definition of the second associator in \eqref{def:sec-asso-294}, we have
	\begin{align*} f(\overline{p} x)
=\overline{p} f(x)+[\overline{p}, x, f].
	\end{align*}
Multiplying the  left and   right above by $p$, we obtain
$$pf(p^{-1} x)p=\Big(p\big(\overline{p} f(x)+[\overline{p}, x, f]\big)\Big)p.$$
We thus can apply the identity
  $$[\overline{p},x,f]=-[p,x,f] $$
and  identity   \eqref{eq:Omod left bi pAp(xf)=Ap(xf)p}
$$p[p,x,f]=[p,x,f]\overline{p}$$
to conclude
\begin{align} \label{eq:var-moufang-110}pf(p^{-1} x)p
=  f(x)p-[p, x, f].
	\end{align}
By the definition of the right scalar multiplication in \eqref{eqdef:<x,fr>}, the right side above is exactly $(f\odot p)(x)$ as desired.
\end{proof}
\begin{rem}
In our typical Example  \ref{eg:Rp}, we have  $\O$-\almost linear function
 $R_q\in  \Hom_{L\O}(\O,\O)$ and its
the right  multiplication by an octonion $p$ is given by
 $$R_q\odot  p=R_{qp}.$$
Indeed,
\begin{eqnarray*}
R_q\odot  p(x)&=&R_q(x) p-[p, x, R_q]
\\ &=&(xq)p-[p, x, q]
\\ &=& (xq)p-[x, q, p]
\\ &=& x(qp)
\\ &=& R_{qp}(x).
\end{eqnarray*}
In this setting,	
	identity \eqref{eq:<x,Tp>=pT(p^-1x)p}   becomes
\begin{eqnarray*}
	{(R_q\odot p)(x)}=pR_q(p^{-1}x)p.
	\end{eqnarray*}
That is,
\begin{equation}\label{eq:variant-MF-944}
  x(qp)=p((p^{-1}x)q)p
\end{equation}
for any $p, q, x\in\O$.
	This is   a variant of the following Moufang identity \cite{schafer2017introduction}
	\begin{eqnarray}
	(xy)(ax)=x(ya)x.
	\end{eqnarray}
Consequently,
  \eqref{eq:<x,Tp>=pT(p^-1x)p}   is a generalized Moufang identity.
\end{rem}

\begin{thm}\label{thm:B(XY) is Banach}
	$H^*$ is a Banach right $\O$-module.
\end{thm}
\begin{proof}
For any $f\in H^*$ and $p\in \O$, it follows from identity \eqref{eq:<x,Tp>=pT(p^-1x)p} that
	\begin{eqnarray*}
	\fsh{f\odot p}&=&\sup_{\fsh{x}
\leqslant1}\fsh{pf(p^{-1}x)p}
\\ &\leqslant&
\sup_{\fsh{x}\leqslant1}\abs{p} \  \fsh{f}\  \| p^{-1}x\| \ \abs{p}
\\ &\leqslant& |p|\fsh{f}.
	\end{eqnarray*}
Thus $f\odot p$ is a continuous $\O$-\almost linear functional, i.e., $f\odot p\in H^*$.  In view of Theorem \ref{Thm:left, bimod Hom(M,M')}, we conclude that $H^*$ is an $\O$-submodule of the right $\O$-module $\Hom(H,\O)$.

	Since $p$ and $f$ are arbitrarily fixed, we can replace $p$ by $p^{-1}$ and   $f$ by $f\odot p$ to get
\begin{eqnarray*} \fsh{(f\odot p)\odot p^{-1}}
 \leqslant  |p^{-1}|\fsh{f\odot p}
 \end{eqnarray*}
so that
	\begin{eqnarray*}\fsh{f}&=&\fsh{(f\odot p)\odot p^{-1}}
\\ &\leqslant& |p^{-1}|\fsh{f\odot p}
\\ &\leqslant& |p^{-1}|\ |p| \  \fsh{f}
\\ &=& \fsh{f}.\end{eqnarray*}
This implies  $$\fsh{f\odot p}=|p|\fsh{f}.$$
%\label{eqdef:funct norm}
Hence $H^*$ is a normed right $\O$-module with respect to  norm \eqref{eqdef:funct norm}.
	
	Finally, we show that $H^*$ is a Banach right $\O$-module.  	
Let	$H^{*_\R}$ stand for the real dual space of  $H$ regarded as a real Hilbert space, i.e.,
$$H^{*_\R}=\{f: H\xlongrightarrow[]{} \mathbb R \ | \ \mbox{$f$ is a  bounded real linear functional} \}.$$
	It is well-known that $H^{*_\R}$ is a real Banach space. Therefore, if $f_n\in H^{*}$ is  a Cauchy sequence, then there  exists $ f\in H^{*_\R}$ such that $\lim _{n\to\infty}f_n=f$. It is easy to check that the maps $$[p,x,\cdot]:H^*\to \O$$ and  $$\re:\O\to \O$$ are both continuous. Hence $$\re [p,x,f]=\re([p,x,\lim _{n\to\infty}f_n])=\re(\lim _{n\to\infty}[p,x,f_n])=\lim _{n\to\infty} \re([p,x,f_n])=0.$$
	This proves $f\in H^{*}$.
\end{proof}

 The following lemma shows that  $H^*$ as a real Hilbert space is isometric with $H^{*_\R}$  via the   real part operator.

\begin{lemma}\label{lem:fsh f=fsh f0} The real part map
\begin{eqnarray*}\re: H^{*}&\xlongrightarrow[]{}& H^{*_\R},
\end{eqnarray*}
defined by $(Re f)(x):=f_{\mathbb R}(x)$,
is a {norm-preserving} map. That is,
	for any $f\in H^{*}$  we have $$\fsh{f}_{H^{*}}=\fsh{f_{\R}}_{H^{*_\R}}.$$
	 \end{lemma}

\begin{proof}
	Let $f\in H^{*}$ and   $x\in H$ with $\fsh{x}\leqslant1$. Then there exists $p\in \spo$ with $|p|=1$ such that $$|f(x)|=pf(x)=\re(pf(x)).
	$$
	Since $f$ is \almost linear, it follows that $$|f(x)|=\re(pf(x))=\re(f(px))=f_{\R}(px)\leqslant\fsh{f_{\R}}_{H^{*_\R}}\fsh{px}\leqslant\fsh{f_{\R}}_{H^{*_\R}}.$$
	Thus we obtain $\fsh{f}_{H^{*}}\leqslant\fsh{f_{\R}}_{H^{*_\R}}$.

On the other hand, it is easy to see that  $\fsh{f_{\R}}_{H^{*_\R}}\leqslant\fsh{f}_{H^{*}}$. This completes the proof.
\end{proof}

\bigskip
To state the  octonionic version of Riesz representation theorem, we need introduce the notion of conjugate $\O$-linearity. %plays an eminent  role in .

\begin{mydef}\label{def:conjugate homo}
	Let $M_1$  be a left $\O$-module and  $M_2$ a right $\O$-module.
	A map $f\in \Hom_\R(M_1,M_2)$ is called   \textbf{conjugate $\O$-linear}  if
	$$f(px)=f(x)\overline{p}$$
	for all $p\in \O$ and all $x\in M_1$.

\end{mydef}

\begin{thm}[\textbf{Riesz representation theorem}]\label{thm:Riesz representation theorem}
	Let $H$ be an   $\O$-Hilbert space.  Then there is an isometric   conjugate  $\spo$-isomorphism
	\begin{eqnarray}\label{eq:riesz iso}
	H^{*}\cong {H}.
	\end{eqnarray}
\end{thm}
\begin{proof}
	Let $f\in H^{*}$ be \almost linear.  By  Theorem \ref{thm: f in Hom(M,M')  equivalent:}, it can be written as  $$f(x)=f_{\R}(x)-\sum_{i=1}^7 e_if_{\R}(e_ix),$$ where $f: H \to \R $ is a real linear  functional. Then  the classical Riesz representation theorem shows  there exists a unique element $z_f\in H$ such
	that $$f_{\R}(x)=\fx{x}{z_f}_\spr$$ and $$ \fsh{z_f}=\fsh{f_{\R}}.$$
Since both $f$ and $\langle \cdot, z_f\rangle$ are $\O$-para-linear and share the same real part, they must coincide as shown by
Corollary \ref{cor: A_p(x,f)=0 and f(px)=pf(x)},
  i.e.,
	$$f(x)= \fx{x}{z_f}.$$
	By Lemma   \ref{lem:fsh f=fsh f0} we have
	$$\fsh{z_f}=\fsh{f}.$$
	Recall that $H^{*}$ is a Banach right $\O$-module by Theorem \ref{thm:B(XY) is Banach}. Hence,  the map \begin{eqnarray*}
		\sigma:H^{*}&\rightarrow& H\\
		f&\mapsto& z_f
	\end{eqnarray*}
is an
	isometry between $H^{*}$ and $H$ as real Banach spaces.
	
	We next prove that $\sigma$ is a conjugate  $\spo$-homomorphism.
	For  any $x\in H$ and any $r\in \spo$,  by definition \eqref{eqdef:<x,fr>} we have
	\begin{eqnarray*}
	\fx{x}{z_{f\odot r}}&=&{(f\odot r)(x)}
\\
	&=&f(x)r-[r,x,f]
 \\
		&=&\fx{x}{z_f}\overline{\overline{r}}+[\overline{r},x,z_f]
\\
	&=&\fx{x}{\overline{r}z_f}.
	\end{eqnarray*}
	This shows that $$ \sigma{(f\odot r)}=\overline{r}\sigma(f).$$ This implies that  $\sigma$ is conjugate  $\spo$-linear.
	
	We finally show that $\sigma$ is a conjugate $\spo$-isomorphism.
	We define \begin{eqnarray*}
		\tau:H&\to &H^{*}\\
		y&\mapsto& y':=\fx{\cdot}{y}.
	\end{eqnarray*}
	We now prove $\tau$ is also a conjugate  $\spo$-homomorphism.
	For any $x\in H$ and any $r\in \O$, it follows from  \eqref{eqdef:<x,fr>}  and \eqref{eq:ass coincide}  that
	\begin{align*}
{(y'\odot r)(x)}&={y'(x)}r-[r,x,y']\\
	&=\fx{x}{y}r-[r,x,y]
	\\
	 &=\fx{x}{y}\overline{\overline{r}}+[\overline{r},x,y]\\
	&=\fx{x}{\overline{r}y}\\
	&={(\overline{r}y)'(x)}.
	\end{align*}
	This shows that $$\tau(y)\odot r=\tau(\overline{r}y),$$ i.e., $\tau$ is also a conjugate  $\spo$-homomorphism.

Moreover, it is easy to check that
	$$\sigma \tau=id,\quad \tau\sigma=id.$$
	Hence $\sigma$ is a  conjugate  $\spo$-isomorphism.
	This finishes the proof.
\end{proof}
{Associated to   a right $\O$-module $M$ is an induced  left $\O$-module
$M^{-}$ with multiplication   defined by
$$p\cdot x:=x\overline{p}$$
 for all $x\in M$ and all $p\in \O$.
We can thus identify two  left modules  $H$ and  $({H^*})^-$ by the isomorphism $\sigma$ in Theorem \ref{thm:Riesz representation theorem}. Moreover,  this isomorphism induces on $({H^*})^-$   a  canonical $\O$-inner product
$$\fx{f}{g}:={\fx{z_f}{z_g}}$$ for any $f,g\in (H^*)^{-}$.
By definition, the octonion  hermiticity and positivity can be verified directly.
  It remains  to show the \almost linearity.
By calculation,  for any $r\in \O$ and $f,g\in H^*$  we have
\begin{eqnarray*}
\fx{r\cdot f}{g}&=&\fx{f\odot \overline{r}}{g}\\
&=&\fx{z_{f\odot \overline{r}}}{z_g}\\
&=&\fx{rz_{f}}{z_g}\\
&=&r\fx{z_{f}}{z_g}+[r,z_{f},z_{g}].
\end{eqnarray*}
This shows that $\fx{\cdot}{g}$ is \almost linear as desired. Then $\sigma$ is an isomorphism between two  Hilbert left $\O$-modules.
}

\bigskip

Based on Theorem \ref{thm:Riesz representation theorem}, we can      characterize  the set
$$H^{*_\O}=\{f: H\xlongrightarrow[]{} \mathbb O \ | \ \mbox{$f$ is a  bounded $\O$-linear functional} \}$$
  in terms of  associative elements of $H$.
To do this, we   need  a criterion for an element $x\in H$ to be an associative element.

\begin{lemma}\label{lem:huaa(H)=huaaer(H)}
	Let $ H $ be a Hilbert left $\O$-module and $x\in H$. Then $x\in \huaa{H}$  if and only if
\begin{equation} \label{eq:assup-ass-392}[p,x,y]=0
\end{equation}
 for all $p\in\spo$ and $y\in H.$
\end{lemma}
\begin{proof}
	If  $x\in \huaa{H}$, then Corollary \ref{cor: A_p(x,f)=0 and f(px)=pf(x)} shows  that $$[p,x,y]=[p,x,y']=0,$$ where $y'$ is defined as in the proof of Theorem \ref{thm:Riesz representation theorem}.

 Conversely, if $x\in H$ satisfies   hypothesis \eqref{eq:assup-ass-392},  then identity \eqref{eq:Re<[p,q,x],y>=Re(pA(x,y))} implies that $$\fx{[q,p,x]}{y}_{\spr}=0$$ for any $p,q\in \spo$ and $y\in H$. Since $\left(H,\fx{\cdot}{\cdot}_{\spr}\right)$ is a  real Hilbert space, we  thus conclude that $$[q,p,x]=0$$ for all $p,q\in \spo$,	which means that $x\in \huaa{H}$. This completes the proof.
\end{proof}
 \begin{thm}\label{thm:riesz real}
	Let $H$ be an $\spo$-Hilbert space. Then there is an isometric isomorphism of real Banach spaces
	\begin{eqnarray}\label{eq:riesz iso 2}
	H^{*_\spo}\cong \huaa{H}.
	\end{eqnarray}
\end{thm}
\begin{proof}
According to Theorem \ref{thm:Riesz representation theorem}, it suffices to show that
$$\tau|_{\huaa{H}}:\huaa{H}\to H^{*_\O}$$ is  surjective.

If $f\in  H^{*_\O}$,  then Theorem \ref{thm:Riesz representation theorem} ensures the existence of   an element $z_f\in H$ such that
$$f(x)=\fx{x}{z_f}.$$
Obviously, $$\tau(z_f)=f.$$ It remains to show that $z_f\in \huaa{H}$.
Since $f$ is $\O$-linear, it follows from identity \eqref{eq:ass coincide} that for any $p\in \O$,
$$[p,z_f,x]=-[p,x,z_f]=-[p,x,f]=-(f(px)-pf(x))=0.$$
Then we conclude from Lemma \ref{lem:huaa(H)=huaaer(H)} that $z_f\in \huaa{H}$.
This completes the proof.
\end{proof}

\begin{rem}
 Isomorphisms \eqref{eq:riesz iso} and \eqref{eq:riesz iso 2} indicate that the objects   in octonionic functional analysis are   $\O$-\almost linear maps rather than $\O$-linear maps.
\end{rem}

	\bibliographystyle{plain}
%\bibliography{hilbanach}

\bigskip\bigskip

\end{document}